\documentclass[a4]{article}
\usepackage[normalem]{ulem}
\usepackage{amsmath,amsthm,amssymb, bbm}
\usepackage{todonotes}
\usepackage{graphicx,xcolor}
\usepackage{algorithm,algpseudocode}
\usepackage[inline, shortlabels]{enumitem}
\usepackage{tikz}
\usetikzlibrary{arrows.meta}
\usetikzlibrary{shapes}
\usepackage[backend=biber, style=ieee, sorting=nyt, citestyle=numeric-comp]{biblatex}

\usepackage{hyperref}
\hypersetup{       
    colorlinks=true,                
    linkcolor=blue,                
    citecolor=blue,
    urlcolor=blue
}

\newcommand{\srccode}{https://github.com/7angel4/betting-game-proof}

\newtheorem{theorem}{Theorem}
\newtheorem*{theorem*}{Theorem}
\newtheorem{lemma}[theorem]{Lemma}
\newtheorem{conjecture}[theorem]{Conjecture}
\newtheorem*{conjecture*}{Conjecture}
\newtheorem{corollary}{Corollary}[theorem]

\providecommand{\keywords}[1]
{
  \small	
  \textbf{\textit{Keywords---}} #1
}

\newcounter{milp}
\renewcommand{\themilp}{\roman{milp}}  % (A, B, C, ...) labelling
\newenvironment{milp}
  {\refstepcounter{milp}\begin{equation}\tag{\themilp}}
  {\end{equation}}

\theoremstyle{definition}
\newtheorem{defn}{Definition}[section]
\DeclareMathOperator*{\argmin}{arg\,min}

\newcommand{\1}{\mathbbm{1}}
\newcommand{\indic}[1]{\1_{\{#1\}}}
\newcommand{\reals}{\mathbb{R}}
\newcommand{\Hset}{\mathcal{H}}
\newcommand{\vect}[1]{\textbf{#1}}
\newcommand{\blank}[1]{}
\newcommand{\meshalgo}{\texttt{MeshItUp}}
\newcommand{\dspace}{\,\,}

\newcommand{\greenbox}[1]{\color{green}\boxed{\color{black}{#1}}\color{black}}
\newcommand{\redbox}[1]{\color{red}\boxed{\color{black}{#1}}\color{black}}

\newcommand{\subsubsubsection}[1]{\textit{#1}.}
\newcommand{\dn}{\Delta_n}
\newcommand{\dnr}[1]{\Delta_{#1}}
\newcommand{\dnR}{\dnr{n,R}}
\newcommand{\alphaval}{(1/2)^n(4/5)}
\title{All In: Give me your money!}
\author{Angel Y. He\footnote{Department of Computer Science, The University of Oxford\\
email:  \texttt{angel.he@balliol.ox.ac.uk}}
\dspace and Mark Holmes\footnote{School of Mathematics and Statistics, The University of Melbourne\\
email:  \texttt{holmes.m@unimelb.edu.au}}}
\date{}

\begin{document}
\maketitle

\begin{abstract}
We present a computer assisted proof for a result concerning a three player betting game, introduced by Angel and Holmes.  The three players start with initial capital $x,y,z>0$ respectively.  At each step of this game two players are selected at random to bet on the outcome of a fair coin toss, with the size of the bet being the largest possible, namely the total capital held by the poorer of the two players at that time.  The main quantity of interest is the probability of player 1 being eliminated (reaching $0$ capital) first. Angel and Holmes have shown that this probability is not monotone decreasing as a function of the initial capital $x$ of player 1.  They conjecture that if $x<y<z$ then player 1 would be better off (less likely to be eliminated first) by swapping their capital with another player. 

In this paper we present a computer-assisted proof of this conjecture.  To achieve this, we introduce the theoretical framework \meshalgo, and then perform a two-stage reduction to make $\meshalgo$ computationally feasible, through the use of mixed-integer programming. 
%In addition to proving the central theorem of interest, we extend this computer-aided proof framework to all theorems of a specific structure.
\end{abstract}


%% COMPLETION BY MARK
\section{Introduction}
Consider a tournament consisting of individual players in which: (i) players exchange money and are eliminated one by one when their capital reaches \$0; and (ii) players receive prize money based on their position in the finishing order.  Suppose that at some point the non-eliminated players decide to share their prize money before the winner has been determined (e.g.~when there are still 3 players left in the game).  How should the prize be split?   The answer should depend on the likelihood of various finishing orders, given the current capital that each of the non-eliminated players has.  

This problem (in the context of poker tournaments) motivated Diaconis and coauthors \cite{diaconis2021gambler,diaconis2022gambler}
to study a multi-player version of the gambler's ruin problem that we present here for just 3 players as follows.  Starting with initial \$ amounts $x,y,z\in \mathbb{N}$, at each step of the process two individuals are chosen uniformly at random to make a \$1 bet on the outcome of a fair coin toss.  The ``fair coin toss'' can be thought of as a manifestation of the players being of equal ability.  A player is eliminated when their capital (typically called a \emph{stack size} in this poker setting) reaches \$0.  The first player eliminated is called the loser, and the player with all the money at the end is declared the winner.  A simple ``Martingale''\footnote{A martingale is a fair game.} argument shows that the probability that player 1 is the winner of this game is $x/(x+y+z)$.  The probability that player 1 is the loser is not an elementary function of $x,y,z$. 

Angel and Holmes \cite{angel2024betting} have introduced a variant of the game above where \$1 bets are replaced with maximal bets: if the two players who go head to head for this round have stack sizes $0<a\le b$ then the bet is for \$$a$.  Such a situation is sometimes referred to as a ``heads up'' (i.e.,~involving only two of the players) ``all in'' (i.e.,~one of the players is betting all of their money) bet in poker.  In this model stack sizes  and bet sizes need not be integers, and the model is scale invariant in that multiplying all of the stack sizes by $c>0$ does not affect the game.   Using a simple martingale argument, it is still the case that 
the probability of player 1 winning is  $x/(x+y+z)$ when the initial stack sizes are $x$ (for player 1), $y$, and $z$ respectively.  Angel and Holmes have shown that the probability $f(x,y,z)$ of player 1 losing is not a decreasing function of $x$ (in the case of \$1 bets above the corresponding function is decreasing in $x$).  Indeed there are situations where if player 1 wants to reduce their probability of being eliminated first then they could achieve this by throwing away some of their stack or giving some to another player.  They conjecture however that a weaker type of monotonicity holds, namely that if $x<y<z$ then to reduce their probability of losing, player 1 would be better off swapping their stack with that of one of the other players.  A somewhat imprecise phrase which captures the distinction between the two statements is the following:
\begin{quotation}
 \emph{  You are not necessarily better off starting with more money, but you are better off starting with more money than your opponent.} 
\end{quotation}

The main contribution of this paper is to present a computer-assisted proof of the aforementioned  conjecture.

%Diaconis and coauthors 
%We study the real-valued version of the ``All in'' model first proposed in \cite{diaconis2022gambler}: Recall the following betting game involving 3 players, starting with stack sizes $x,y,z$.  For each round of the game, two players are picked uniformly at random to bet on the outcome of a fair coin toss.  The size of the bet is equal to the minimum of their two current stacks.  As soon as any player reaches \$0 they are declared the loser.  Let $f(x,y,z)$ be the probability that player 1 is the loser if the starting stacks are $x,y,z$. 

%% ...
%% Motivation, impacts...
%% Why is it important, why are we studying this?

%\subsubsubsection{Our results}
%We seek to better understand the effect of the players' initial capital (stack size) on their chance of \textit{losing}, i.e., being the first to be eliminated from the game. Our main result approximately says the following:

\begin{theorem}
\label{thm:want_to_swap}
If $0<x<y<z$ then $f(x,y,z)>f(y,x,z)>f(z,x,y)$.  In particular,  
 player 1 is less likely to lose if they swap their initial stack with that of another player who has a larger initial stack.
\end{theorem}
Note that for any $x,y,z$, $f(x,y,z)=f(x,z,y)$ as by symmetry it makes no difference to player 1 which of the two other players has $y$ and which has $z$.

As the game progresses, the stack sizes evolve as a discrete-time (time-homogeneous) Markov chain taking values in the 3-simplex, because the sum of the three stack sizes is preserved throughout the game\footnote{That is, the 3-simplex is defined by $\{(x,y,z): x,y,z>0, x+y+z=1\}$, and we specifically focus on its subset where $x<y<z$. Since the game is scale invariant, we can use this normalized representation of stack sizes instead of constraining their absolute values.}.  From each state with strictly positive stack sizes there are 6 possible transitions, corresponding to the 6 possibilities for which player wins money from which player in the next round.    Figure \ref{fig:dtmc} shows the 1-step transition probabilities from a state $(x,y,z)$ with $x\le y\le z$.

\begin{figure}[H]
    \centering
    \resizebox{1\textwidth}{!}{ %%%% DTMC
\begin{tikzpicture}[
    every node/.style={draw, ellipse, minimum width=1cm, minimum height=0.5cm, font=\Large},
    every edge/.append style={->, -{Latex[length=2mm,width=5mm]}, thick, line width=1pt}
]

% Main node (x,y,z)
\node (xyz) at (0,0) {\(x,y,z\)};

% Other nodes evenly spread below
\node (1) at (-10,-3) {\(0,y+x,z\)};
\node (2) at (-6,-3) {\(2x,y-x,z\)};
\node (3) at (-2,-3) {\(0,y,z+x\)};
\node (4) at (2,-3) {\(2x,y,z-x\)};
\node (5) at (6,-3) {\(x,0,z+y\)};
\node (6) at (10,-3) {\(x,2y,z-y\)};

% Edges from (x,y,z) to nodes 1-6, with labels 1/6
\draw[->] (xyz) edge node[right, draw=none, above=1pt] {\Large \( \frac{1}{6}\)} (1);
\draw[->] (xyz) edge node[left,draw=none, left=5pt] {\Large \(\frac{1}{6}\)} (2);
\draw[->] (xyz) edge node[left, draw=none] {\Large \(\frac{1}{6}\)} (3);
\draw[->] (xyz) edge node[right, draw=none] {\Large \(\frac{1}{6}\)} (4);
\draw[->] (xyz) edge node[right, draw=none, right=7pt] {\Large \(\frac{1}{6}\)} (5);
\draw[->] (xyz) edge node[right, draw=none, above=1pt] {\Large \(\frac{1}{6}\)} (6);

\end{tikzpicture} }
    \caption{Transition probabilities from a state $(x,y,z)$ with $0\le x\le y\le z$.}
    \label{fig:dtmc}
\end{figure}
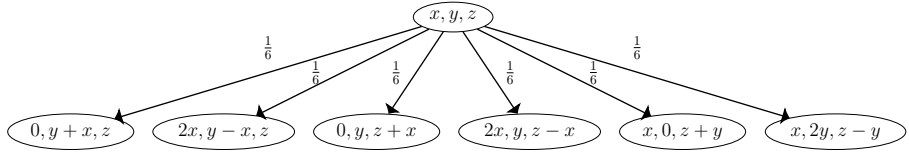
For some transitions from this state a player reaches 0, so the loser has been determined and the game ends.  For others it can be that all three players still have strictly positive stack sizes, so the game continues.  Note that if $x=y=z$ then the game is guaranteed to end after just one round, since no matter which pair of players compete in the next round, the player who loses the bet reaches 0.  

\textit{Overview}. 
In Section \ref{sec:the_alg} we present some preliminary calculations that motivate the form of the algorithm \meshalgo, which is presented at the end of the section.  The
%We introduce  an
algorithm
%, 
%\meshalgo, which 
involves partitioning the 3-simplex into numerous regions, and checking whether a specific inequality %\eqref{eqn:mesh-ineq}
holds within each region.  While this algorithm is designed to theoretically establish Theorem \ref{thm:want_to_swap}, we need to adopt 
%it is computationally infeasible in its initial form. To address this, we propose 
a two-step reduction process to make it computationally feasible. Specifically, we reformulate the core problem as a constrained optimization problem, and then (by converting the strict inequality constraints into non-strict inequality constraints) transform it into a solvable mixed-integer linear program.  This is described in the majority of Section \ref{sec:mip}. 
The source code of our implementation\footnote{The implementation is entirely the work of the first author.} (in Python) for the computer-assisted proof is available at \url{\srccode}. In Section \ref{sec:adapt-algo} we provide a sketch of how this hybrid mathematical-computational framework could be adapted to prove a hypothetical inequality of a similar form to Theorem \ref{thm:want_to_swap}.

\blank{
\subsubsubsection{Overview}
We give an outline of the remainder of this work.
\begin{itemize}
    \item In Section \ref{sec:background}, we present some prior work and background concepts relevant to the proof.
    \item In Section \ref{sec:bigpic}, we present the big picture of the proof for Theorem \ref{thm:informal_want_to_swap}, including the \meshalgo, algorithm.
    \item In Section \ref{sec:mip}, we present the mixed-integer programming framework that enables the implementation of \meshalgo.
    \item In Section \ref{sec:results}, we present the main results based on our implementation, which explains the ``magic numbers'' in Section \ref{sec:bigpic}.
\end{itemize}
}

\section{Preliminary Calculations}
\label{sec:the_alg}
In this section we present some elementary calculations that establish the theoretical basis for the proof, and present the $\meshalgo$ algorithm.  The first result gives a uniform upper bound on the probability of player 1 being the loser.

\begin{lemma}
For any $x,y,z>0$, $f(x,y,z)\le \frac45$ and $f(x,z,z)\le \frac{2}{3}$.
\end{lemma}

\begin{proof}
Note that for any $x,y,z>0$ we have that 
\begin{align}
1-f(x,y,z)\ge \sum_{n=1}^\infty \left(\frac16 \right)^n = \frac15
\label{lem:frac15}
\end{align}
as the right hand side is (a lower bound on) the probability that one of the other two players eliminates the other without player 1 even playing.  The first claim follows.  

For the second claim simply note that if player 1 is not selected  for the first round (this has probability $1/3$) then they are not the loser.
%As a consequence, for any $x,y,z>0$
%\begin{align}
%f(x,y,z)\le \frac45,\label{45}
%\end{align}
%as claimed.
\end{proof}

Let $h_n(x,y,z)$ be the probability that player 1 is eliminated in exactly round $n$ when the initial stacks are $x,y,z$.  Then 
\begin{equation}
h_1(x,y,z)=\frac{\indic{x,y,z>0}}{6}\Big[\indic{x\le y}+\indic{x\le z}\Big]. 
\label{h1} 
\end{equation}

Moreover, for general $a,b,c$ and $n\ge 2$ we have 
\begin{align}
h_n(a,b,c)&=\frac{\indic{a,b,c>0}}{6}\Big[h_{n-1}(2a,b-a,c)+h_{n-1}(2a,b,c-a)\nonumber\\
&\phantom{=\frac{\indic{a,b,c>0}}6\Big[}+h_{n-1}(a-b,2b,c)+h_{n-1}(a,2b,c-b)\nonumber\\
&\phantom{=\frac{\indic{a,b,c>0}}6\Big[} +h_{n-1}(a-c,b,2c)+h_{n-1}(a,b-c,2c)\Big].
\label{hrecursion}
\end{align}
Appendix \ref{sec:hrec-eg} provides an example of unrolling $h_2$ (i.e., expanding the recurrence).

\begin{lemma}
\label{lem:use_me}
    Let $x_1,y_1,z_1,x_2,y_2,z_2>0$.  Then 
\begin{align*}
f(x_1,y_1,z_1)-f(x_2,y_2,z_2)
&\ge \sum_{j=1}^n (h_j(x_1,y_1,z_1)-h_j(x_2,y_2,z_2))-(1/2)^{n}(4/5).
%\label{use_me}
\end{align*}
\end{lemma}
\begin{proof}
Note that 
\begin{align}
f(x_1,y_1,z_1)&\ge \sum_{j=1}^n h_j(x_1,y_1,z_1)\\
f(x_2,y_2,z_2)&\le \sum_{j=1}^n h_j(x_2,y_2,z_2)+(1/2)^n(4/5)
\end{align}
where the last term comes from the possibility that no one has lost in the first $n$ rounds (probability at most $(1/2)^n$) and then at some point thereafter player 1  loses (probability at most $4/5$ by Lemma \ref{lem:frac15}).
\end{proof}

Henceforth we will restrict ourselves to consider only the region appearing in Theorem \ref{thm:want_to_swap}, i.e.
\begin{equation}
\label{eqn:V}
V := \{(x,y,z):{0<x<y<z} \}
\end{equation}
The expressions \eqref{h1} and \eqref{hrecursion} thus show that $h_n(a,b,c)$ can be expressed as a finite sum of indicators of subsets of the 3-simplex times $(1/6)^n$.  
% Our main result (Theorem \ref{thm:informal_want_to_swap} formally stated) is the following:
% \begin{theorem}
% \label{thm:want_to_swap} 
% For $0<x<y<z$,  
% %the following are true
% \begin{align}
% f(x,y,z)>f(y,x,z)>f(z,x,y).
% \end{align}
% %\item If $x\le y<z$ then 
% %$f(y,x,z)>f(z,x,y)<f(y,x,z)$
% \end{theorem}
% This says that if another player starts with more money than you then you are more likely (than them) to be the loser.
% Note that in our definition of $f$, the second and third positional argument is unordered, i.e., $f(y,x,z)=f(y,z,x)$.\\
Let 
\begin{equation}
 \dn(x,y,z):=\sum_{j=1}^n h_j(x,y,z)-\sum_{j=1}^n h_j(y,x,z).
\label{eqn:dn}   
\end{equation}
It follows that for any fixed $n$ there are only finitely many different values for  $\dn$, %varies depending on the region in $\mathbb{R}^3$ under consideration, e.g., we might restrict ourselves to the region $\{(x,y,z):{0<x<y<z},\dspace {x+y \geq z}\}$. Thus, 
so $\dn$ can be viewed as a function defined on finitely many regions that form a partitioning of the 3-simplex. 
%containing the point $(x,y,z)$.
%on which it is evaluated. 
We will use $\dnR$ to explicitly denote the value of $\dn$ evaluated on region $R$.
% \todo[inline]{I moved the $x+y+z=1$ to the footnote as above -- I think it's better not to add this into $V$, to avoid confusion with the actual condition specified in the theorem; plus later on I mentioned how $V$ should not contain a free-floating constant term in its inequalities when adapting the proof.}
% \item $\dn>c$ means $\forall R \subset V: \dnR > c$, where $V \subset \mathbb{R}^3$ represents a specific region.\todo[inline]{I'm not really sure what is being said here.  For each $n$ there is a collection of relevant regions (that form a partition of the simplex or of $\{(x,y,z):0<x,y,z\}$) and I guess you are making a statement for all regions in that collection (the collection depends on $n$ I guess?)  But isn't the statement that you are making simply equivalent to $\Delta_n(x,y,z)>c$ for every $(x,y,z)$ in $V$?}

Theorem \ref{thm:want_to_swap} follows from the following two Lemmas, which we will verify via a computer-assisted proof explained in Section \ref{sec:mip}.
\begin{lemma}
\label{lem:swap_xy1}
For all $0<x<y<z$, $\Delta_4(x,y,z)\ge 
\redbox{\dfrac{43}{648}}
>\frac{1}{20} =(1/2)^4(4/5)$.
  \end{lemma}

\begin{lemma}
\label{lem:swap_yz1}
For all $0<x<y<z$, $\Delta_4(y,z,x)\ge 
\redbox{\dfrac{37}{648}}
>\frac{1}{20}=(1/2)^4(4/5)$.
\end{lemma}

\begin{proof}[Proof of Theorem \ref{thm:want_to_swap}]
Let $0<x<y<z$.  From Lemmas \ref{lem:use_me} and \ref{lem:swap_xy1} we have that 
\begin{align}
f(x,y,z)-f(y,x,z)\ge \Delta_4(x,y,z)-(1/2)^4(4/5)>0.
\end{align}
This proves the first claim of the theorem.

From Lemmas \ref{lem:use_me} and \ref{lem:swap_yz1} we have that 
\begin{align}
f(y,z,x)-f(z,y,x)\ge \Delta_4(z,y,x)-(1/2)^4(4/5)>0.
\end{align}
This proves the second claim.
\end{proof}

\subsection{The \text{\meshalgo } Algorithm}
\label{sec:mesh}

% For $0<x\le y\le z$ we have the base cases from \eqref{h1}:
% \begin{align}
% h_1(x,y,z)&=\frac13\\
% h_1(y,x,z)&=\frac16+\frac16\indic{y=x}\\
% h_1(z,x,y)&=\frac16\indic{z=y}+\frac16\indic{z=x}.
% \end{align}

Equation \eqref{hrecursion} shows how each of the $h_i$'s can be unrolled into a linear combination of indicator variables. This gives rise to a method to prove Lemma \ref{lem:swap_xy1} and \ref{lem:swap_yz1} with computer assistance as follows:

\newcommand{\indfn}[1]{\texttt{indicators}{\left(#1\right)}}
\begin{algorithm}[H]
\caption{\meshalgo}\label{alg:mesh}
\begin{algorithmic}[1]
    \State \textbf{Input: }$V = \{(x,y,z):0<x<y<z\}$
    \State \textbf{Output: }$n$
    \State $n \gets 2$\
    \While {true}
        \State $R_1, ..., R_{a_n} \gets \indfn{\sum_{i=1}^n{h_j(x,y,z)}}$
        \State $S_1, ..., S_{a_n} \gets \indfn{\sum_{i=1}^n{h_j(y,x,z)}}$
        \ForAll {region $R_j \cap S_k$}
            \State $dn \gets \hyperref[eqn:dn]{\dnr{n, V \cap R_j \cap S_k}}$
            \If{$dn \leq \alphaval$}
                \State $n \gets n+1$
                \State \textbf{break}
            \EndIf
        \EndFor
    \EndWhile
    \State \textbf{return} $n$
\end{algorithmic}
\end{algorithm}
where the $\indfn{\cdot}$ function extracts the indicators from its argument. To prove Theorem \ref{thm:want_to_swap}, we default the input $V$ to its definition in \eqref{eqn:V}.

As aforementioned, the expressions $\sum_{i=1}^n{h_i(\cdot)}$ are just sums of real numbers times indicators of subsets of the 3-simplex. This divides the 3-simplex into $a_n<\infty$ regions. 
Thus, one can think of each $h_i(\cdot)$ as a layer of partitioning of $V$, and $\dn$ (which is a linear combination of the $h_i$'s) as multiple layers of partitionings superimposed on each other. This forms a ``mesh'' on the 3-simplex. As $n$ increases, the ``mesh'' gets finer and finer.

The algorithm essentially tries to find the first $n$ for which $\dn > \alphaval$, i.e., $\forall R \subset V: \dnR > \alphaval$. Its correctness follows from the exhaustive evaluation of $\dn$ on all subregions (``mesh-grids'') in $V$ for a given $n$, until it finds a region $V \cap R_j \cap S_k$ for which 
\begin{equation} 
\dn \leq \alphaval. 
\label{eqn:mesh-ineq}
\end{equation}
If we can find the $n$ at which the algorithm terminates, then we can deduce the ``magic numbers'' (highlighted) in Lemma \ref{lem:swap_xy1} and \ref{lem:swap_yz1}, and complete the proof for Theorem \ref{thm:want_to_swap}.

\section{Mixed-integer Linear Programming}
\label{sec:background}
As we will explain in Section \ref{sec:mip}, our proof employs a mixed-integer linear programming (MILP) framework to establish a lower bound on the difference in player 1’s elimination probability when starting with a stack size of $x$ versus $y$. 
In this section, we introduce the necessary terminology and describe several algorithms for solving MILPs.

A \textit{linear program} (LP) is an optimization problem where the objective function and all constraints are linear. The goal is to maximize/minimize a linear objective function subject to a set of linear inequality constraints, with all decision variables being continuous.

A \textit{mixed-integer program} (MIP) extends linear programming with the additional constraint that some of the decision variables are restricted to integer values.

A \textit{mixed-integer linear program} (MILP) combines both ideas: It is a specific class of MIPs with only linear constraints and a linear objective function. Formally:
\begin{defn}
    Given a matrix $A\in \reals^{m \times n}$, vectors $\vect{b} \in \reals^m$ and $\vect{c} \in \reals^n$, and a non-empty subset $I \subset \{1,\hdots,n\}$, the \textit{mixed-integer linear program} $\text{MILP}=(A,b,c,I)$ is to solve:
    \begin{equation*}
    \label{eqn:milp}
        \vect{c}^* = 
        \min{\{ \vect{c}^T \vect{x} \mid 
        A \vect{x} \leq \vect{b},\, 
        \vect{x} \in \reals^n,\, 
        x_j \in \mathbb{Z} \dspace\forall j \in I 
        \}}
    \end{equation*}
    The vectors $\vect{x}$ that satisfy all constraints are called \textit{feasible solutions} of the MILP. 
    A feasible solution $\vect{x}^*$ is \textit{optimal} if its objective value $\vect{c}^T \vect{x}^* = \vect{c}^*$ \cite{achterberg2007constraint}.
\end{defn}
The integrality constraints make MILPs useful for modelling the discrete nature of some variables, such as an indicator variable whose values are restricted to 0 or 1.

\subsection{Algorithms for solving MILPs}

\paragraph{Branch-and-bound}
The \textit{branch-and-bound} procedure \cite{land1960automatic} is widely used for solving all sorts of optimization problems. It systematically explores the solution space by dividing the given problem instance into smaller subproblems (branching), and using upper and lower bounds to prune suboptimal solutions (bounding). By eliminating regions of the search space that cannot contain the optimal solution, branch-and-bound avoids a complete enumeration of all potential solutions, while guaranteeing that the global optimum is found.

\paragraph{Cutting Planes}
The \textit{cutting-plane} technique \cite{gomory2010outline, gomory1960solving, gomory1963analog} is typically used only for MIPs. It begins by solving a \textit{relaxation} of the problem, where integrality constraints are ignored, then incrementally adds linear constraints (cutting planes) of the form $\vect{a}^T \vect{x} \leq \vect{b}$ to exclude fractional solutions, while retaining feasible integer points. These cuts progressively tighten the feasible region until an optimal integer solution is reached.

\paragraph{Branch-and-cut}
The \textit{branch-and-cut} method \cite{padberg1991branch} combines the strength of the branch-and-bound and the cutting-plane algorithm: It solves the problems with branch-and-bound, while incorporating cutting planes to tighten the feasible region and eliminate fractional solutions. By integrating these two approaches, branch-and-cut efficiently narrows the search space and accelerates the convergence to an optimal integer solution.

Our implementation of the MILP model for $\meshalgo$ uses the branch-and-cut algorithm. 
A theoretical guarantee of this algorithm is that, if we are minimizing but cannot prove optimality of the found solution, then a \textit{lower bound} on the optimal value can be deduced from the algorithm \cite{mitchell2002branch}. 
This ensures the correctness of the proof.

% The branch-and-bound algorithm with cutting planes, to enable efficient exploration of the solution space and tightening the relaxation of the MIP formulation \cite{lodi}. Therefore, to understand how it works, we shall first understand how the constituent methods work. Branch-and-bound is a tree-based approach, where the problem is recursively divided into smaller subproblems by branching on decision variables, thus creating a search tree that is systematically explored. This method guarantees optimality by exhaustively evaluating all potential solutions within the feasible region. However, it is computationally expensive due to the exponential growth of subproblems.

% On the contrary, cutting planes, originally proposed by Gomory, are linear inequalities added to the relaxation of the problem to exclude infeasible integer solutions without eliminating feasible ones, thus tightening the bounds used in the branch-and-bound process. By integrating it with branch-and-bound, the branch-and-cut algorithms allows for more effective pruning of the search tree, while preserving the optimality of the algorithm 
% The combination has proven particularly effective in solving large-scale MIPs that are otherwise intractable using simpler methods.

\blank{\section{\meshalgo: The General Proof Framework}
\label{sec:bigpic}
In this section, we derive the theoretical proof framework, \meshalgo.

\subsection{Preliminary proofs}
}

\section{Implementing \meshalgo}
\label{sec:mip}
%% MOTIVATION
The \hyperref[sec:mesh]{\meshalgo} algorithm introduces two computational challenges:
\begin{enumerate}
    \item \textit{Representation of $\dn$}: How to represent $\dn$ (which is a linear combination of indicators) in program?
    \item \textit{Exponential complexity of \meshalgo}: Each state has at most 6 parent states \eqref{hrecursion}. Thus, each $h_n$ corresponds to at most $k_n := 6^n$ regions, so each $\dn$ corresponds to at most $2\cdot \sum_{i=1}^n{k_i} = \frac{12}{5} \cdot (6^n - 1) = O(6^n)$ regions. This means that the algorithm will become quite slow for even moderate values of $n$ (e.g., $n \geq 6$), if we naïvely check through every subregion of $V$.
\end{enumerate}
In this section we describe our approach to handling the two issues, using a mixed-integer linear programming framework.

\subsection{The MILP Framework}
\label{sec:milp-repr}
The coefficient of all indicators in $h_n$ is the same, which means $h_n(x,y,z)$ could be represented as:
\begin{enumerate}
    \item a single coefficient $c_n := \left(\frac{1}{6}\right)^n$, and 
    \item a set $\Hset_n(x,y,z)$ of indicator constraints.
\end{enumerate}
Therefore, each $\dn$ consists of a set $\Hset_n^+ = \bigcup_{i=1}^n{\Hset_i(x,y,z)}$ of ``positive'' inequalities (from the positive indicators in \hyperref[eqn:dn]{$\dn$}), and a set $\Hset_n^- = \bigcup_{i=1}^n{\Hset_i(y,x,z)}$ of ``negative'' inequalities (from the negative indicators in $\dn$). 
Appendix \ref{sec:hrec-eg} includes the example for $\Delta_2$.

Note that an indicator can represent multiple inequalities, so $\Hset_n$ could be a two-dimensional list of linear inequalities in code. We also assume without loss of generality, that all indicators in $\Hset_n$ are for non-empty regions, since an indicator for an empty region, e.g., $\{(x,y,z): x<y \land x \geq y\}$, can be directly eliminated to 0 and would not contribute to $h_n$.

\subsubsection{Towards Constrained Optimization}
For simplicity of reference, from now on we denote $\alpha_n := \alphaval$.

Consider the body of the while loop in \hyperref[alg:mesh]{\meshalgo}, which essentially searches for a (non-empty) \textit{region} $R$ in $V$ for which $\dnR \leq \alpha_n$.
An alternative way to think of this problem is as follows: 

\begin{quotation}
    \centering 
    \emph{Does there exist a point in $V$ for which $\dn \leq \alpha_n$?} 
\end{quotation}
If we could find such a point, we would immediately know that the current value of $n$ fails, allowing us to proceed to the next $n$ value.
 
However, if at a given point $\dn > \alpha_n$, we cannot conclude that this inequality holds for the entire region $V$. 
This implies that what we actually need is the point where $\dn$ attains its minimum value. Note that there may be multiple such points, as all points within the same subregion (a ``mesh-grid''s) of $V$ will share the same $\dn$ value.

This motivates the formulation of a constrained optimization problem, where we try to minimize the value of $\dn(x,y,z)$ across the region $V$:
\begin{lemma}
\label{cor:milp_red}
    $\exists R \subset V: \Delta_{n,R}(x,y,z) \leq \alpha_n
    \iff 
    \min_{R \subset V}{\dn(x,y,z)} \leq \alpha_n$
\end{lemma}
\begin{proof}~
    \begin{enumerate}
    \item[$(\Rightarrow)$] 
    Suppose $\Delta_{n,R_1}(x,y,z)\leq \alpha_n$, then $\min_{R \subset V}{\dnR(x,y,z)} \leq \Delta_{n,R_1}(x,y,z)\leq \alpha_n$.
    \item[$(\Leftarrow)$]  
    Suppose 
    $\min_{R \subset V}{\dn(x,y,z)} \leq \alpha_n$, then $\dnr{n,R^*}(x,y,z) \leq \alpha_n$, where \\${ R^* = \argmin_{R \subset V}{\dnR(x,y,z)} }.$
    \end{enumerate}
\end{proof}
\begin{corollary}
    $\min_{R \subset V}{\dnR(x,y,z)} > \alpha_n \Rightarrow \dnR(x,y,z) > \alpha_n$.
\end{corollary}
\begin{proof}
    Take the contra-positive of the forward implication $(\Rightarrow)$ in Lemma \ref{cor:milp_red} above.
\end{proof}

\subsubsection{Towards MILP}
As aforementioned in Section \ref{sec:background}, MILP problems involve optimization scenarios where some variables are required to take integer values, which makes it easy to model the indicator variables in $\dn$. We can thus frame the constrained optimization problem as an MILP as follows:

\begin{milp}
\label{milp:origLP}
    \begin{array}{ll@{}ll}
    \text{variables} & x,y,z \in \reals,  \{i: i \in \Hset_n^+ \cup \Hset_n^-\}\\
    \text{minimize}  & \dn(x,y,z) &\\
    \text{subject to}& \indic{c} = 1 \Rightarrow c  &\text{(indicator constraints)}\\
                     & \indic{c} = 0 \Rightarrow \neg c\\
                     & 0<x<y<z & (V) \\
                     & \dn(x,y,z) - \alpha_n \leq 0
    \end{array}
\end{milp}

where $\alpha_n = \left(\frac{1}{2}\right)^n \left(\frac{4}{5}\right)$.\\

Note that there exists methods that can express the indicator constraints as linear constraints, such as the \textit{Big-M} method \cite{dantzig1948programming,soleimani2008modified,bazaraa2011linear,cococcioni2021big}.

\begin{lemma}
    For a given $n$ and non-empty region $R\subset V$, the above MILP \eqref{milp:origLP} is infeasible if and only if $\dn> \alpha_n$.
\end{lemma} 
\begin{proof}~
    \begin{enumerate}
        \item[$(\Rightarrow)$]
        If the MILP is infeasible, then by definition $\nexists (x',y',z'): 0<x'<y'<z' \,\land\, \dnr{n,\{(x',y',z')\}} \leq \alpha_n$, which means $\nexists R \subset V: \dnR \leq \alpha_n$. So $\forall R \subset V: \Delta_{n,R}(x,y,z) > \alpha_n$ and the right-hand side holds.
        \item[$(\Leftarrow)$]
        If $\dn> \alpha_n$, i.e., $\forall R \subset V: \dnR > \alpha_n$, then there exists no non-empty region $R$ and hence no valid coordinates $(x',y',z')$, such that ${0<x'<y'<z'}$ and ${\dnr{n,\{(x',y',z')\}} \leq \alpha_n}$. Thus the MILP has no solution.
    \end{enumerate}
\end{proof}
As aforementioned in Section \ref{sec:background}, the branch-and-cut algorithm used by our implementation guarantees to output an objective value that is a lower bound of the true optimal objective value. Thus, the output objective value will be an absolute lower bound on $\dn$ (over $V$).

\subsubsection{Towards a solvable MILP}
\label{sec:second-reduction}
The previous MILP \eqref{milp:origLP} introduces another problem: Both $V$ and the indicator constraints may contain strict inequalities, i.e. $<$ or $>$. However, strict inequalities are in general, not permitted in LPs, because with such constraints, the feasible set is no longer closed \cite{bertsimas1997introduction}.
That is, for any solution within the feasible region, we can always find a better solution also within the feasible region by moving (infinitesimally) closer to a boundary.

Fortunately, there are two properties of the problem that we can use to resolve this issue:
\begin{enumerate}
    \item \textbf{Scale invariance}: $\forall k \in \reals^+: \dn(kx,ky,kz) = \dn(x,y,z)$ \dspace\cite{angel2024betting}
    \item \textbf{The constraints do not contain any ``free floating'' constants.} 
    In other words, all inequalities in the constraints can be written in the form:
    \begin{equation}
    \label{eqn:linear-combo-constraint}
        ax+by+cz \oplus 0
    \end{equation}
    where $a,b,c \in \mathbb{Z}$, and $\oplus \in \{<,>,\leq,\geq\}$ is a comparison operator.
\end{enumerate}

We now show that the previous MILP \eqref{milp:origLP} can be reduced to an MILP involving non-strict inequalities only, by converting all strict inequalities of the form (without loss of generality for $<$):
\begin{equation*}
    ax+by+cz>0
\end{equation*}
into
\begin{milp}
\label{milp: reducedLP}
    ax+by+cz \geq \epsilon
\end{milp}
where $\epsilon \in \reals^+$.\\

\newcommand{\origopt}{\Delta^*_1}
\newcommand{\redopt}{\Delta^*_2}
Let $\origopt$ be the optimal objective value for MILP \eqref{milp:origLP}, and $\redopt$ be the optimal objective value for the resulting MILP from \eqref{milp: reducedLP}.
We prove that the reduction is valid by proving the following equality:
\begin{theorem}
    $\origopt = \redopt$
\label{thm:eq-opt-obj}
\end{theorem}

\begin{proof}~
\begin{enumerate}
    \item[$(\Rightarrow)$]$\origopt \leq \redopt$

    Since $ax+by+cz+d \geq \epsilon \Rightarrow ax+by+cz+d>0$, and the other constraints in MILP \eqref{milp:origLP} remain unchanged in MILP \eqref{milp: reducedLP}, any valid solution to MILP \eqref{milp: reducedLP} must also be a valid solution to MILP \eqref{milp:origLP}..
    
    Therefore, if $(x^, y^, z^*)$ is a valid solution to MILP \eqref{milp: reducedLP}, it must also be valid for MILP \eqref{milp:origLP}. Consequently, $\redopt$ must be a feasible objective value of MILP \eqref{milp:origLP}, as both MILPs share the same objective. By the optimality assumption for $\origopt$ (in a minimization MILP), we have $\origopt \leq \redopt$.

    \item[$(\Leftarrow)$] $\origopt \ge \redopt$
    
    Suppose $(x^*, y^*, z^*)$ is an optimal solution to MILP \eqref{milp:origLP} that produces $\origopt$. Denote the following:
    \begin{itemize}
        \item $S := $ the set of strict inequalities, each of the form $ax+by+cz>0$, satisfied by $(x^*, y^*, z^*)$. Note that $S$ also includes the negation of those non-strict inequalities violated by $(x^*, y^*, z^*)$.
        \item $N := $ the set of non-strict inequalities, each of the form ${ax+by+cz \geq 0}$, satisfied by $(x^*, y^*, z^*)$. Note that $N$ also includes the negation of those strict inequalities violated by $(x^*, y^*, z^*)$.
    \end{itemize}

We can write $S$ as:
\begin{equation*}
    S=
    \left\{\!\begin{aligned}
        a_1 x^* + b_1 y^* &+ c_1 z^* > 0\\[1ex]
        &\vdots\\[1ex]
        a_j x^* + b_j y^* &+ c_j z^* > 0
    \end{aligned}\right\}
    =
    \left\{\!\begin{aligned}
        a_1 x^* + b_1 y^* &+ c_1 z^* \geq \epsilon_1\\[1ex]
        &\vdots\\[1ex]
        a_j x^* + b_j y^* &+ c_j z^* \geq \epsilon_j
    \end{aligned}\right\}
\end{equation*}
for some $\epsilon_1, ..., \epsilon_j \in \reals^+$.

We can similarly write $N$ as:
\begin{equation*}
    N =
    \left\{\!\begin{aligned}
        a_1 x^* + b_1 y^* &+ c_1 z^* \geq 0\\[1ex]
        &\vdots\\[1ex]
        a_j x^* + b_j y^* &+ c_j z^* \geq 0
    \end{aligned}\right\}
\end{equation*}

Now construct a solution to MILP \eqref{milp: reducedLP}, by scaling $(x^*, y^*, z^*)$ by 
\[
{k := \max{\left\{\frac{\epsilon}{\epsilon_1}, \ldots, \frac{\epsilon}{\epsilon_j}\right\}} \in \mathbb{R^+}}
\]
so that
${k \cdot \epsilon_i \geq \epsilon\dspace \forall i \in \{1, \ldots, j\}}$.

Therefore, the new solution $(kx^*, ky^*, kz^*)$ satisfies the following sets of inequalities:
\begin{equation*}
    {S' =
    \left\{\!\begin{aligned}
        a_1 kx^* + b_1 ky^* &+ c_1 kz^* \geq \epsilon\\[1ex]
        &\vdots\\[1ex]
        a_j kx^* + b_j ky^* &+ c_j kz^* \geq \epsilon
    \end{aligned}\right\}
    }
    \quad \text{and} \quad
    {N' =
    \left\{\!\begin{aligned}
        a_1 kx^* + b_1 ky^* &+ c_1 kz^* \geq 0\\[1ex]
        &\vdots\\[1ex]
        a_j kx^* + b_j ky^* &+ c_j kz^* \geq 0
    \end{aligned}\right\}
    }
\end{equation*}

This implies that $(kx^*, ky^*, kz^*)$ satisfies the corresponding set of constraints in MILP \eqref{milp: reducedLP}, where the strict inequalities in MILP \eqref{milp:origLP} are transformed into non-strict inequalities. Therefore, $(kx^*, ky^*, kz^*)$ is a valid solution to MILP \eqref{milp: reducedLP}.

Moreover, by scale invariance, $(kx^*, ky^*, kz^*)$ will yield the same objective value in MILP \eqref{milp: reducedLP} as $(x^*, y^*, z^*)$ does in MILP \eqref{milp:origLP}.

In other words, from the optimal solution of MILP \eqref{milp:origLP}, which gives the objective value $\origopt$, we can construct a valid solution to MILP \eqref{milp: reducedLP} with the same objective value $\origopt$. Since we assumed $\Delta_2^*$ to be the optimal objective value of MILP \eqref{milp: reducedLP}, it must be at least as good as any other possible objective value, meaning $\redopt \leq \origopt$.
\end{enumerate}

Therefore, $\origopt = \redopt$.
\end{proof}

Note a consequence of the theorem above: 
\begin{corollary}
    There exists a valid solution to MILP \eqref{milp:origLP} if and only if there exists a valid solution to MILP \eqref{milp: reducedLP}.
\end{corollary}
\begin{proof}
    In the proof for Theorem \ref{thm:eq-opt-obj} above, we have shown that from a solution to MILP \eqref{milp:origLP}, we can construct a solution to MILP \eqref{milp: reducedLP} (yielding the same objective value), and vice versa.
\end{proof}
So if MILP \eqref{milp: reducedLP} is infeasible, we know that MILP \eqref{milp:origLP} must also be infeasible.
With this solvable MILP \eqref{milp: reducedLP} which consists of only non-strict inequality constraints, we can implement $\meshalgo$ in code to complete the proof for Lemmas \ref{lem:swap_xy1} and \ref{lem:swap_yz1}, and thus for Theorem \ref{thm:want_to_swap}. 

However, we can make one further simplification, as we will explain in the next section.

\subsubsection{Towards easier-to-solve MILPs}
\label{sec:distr-opt}
We could rewrite:
\begin{align}
    \dn(x,y,z) - \alpha_n 
    &= \sum_{j=1}^{n}{h_j(x,y,z)} - \sum_{j=1}^{n}{h_j(y,x,z)} - \left(\frac{1}{2}\right)^n \left(\frac{4}{5}\right) \\
    {}&=
    \sum_{j=1}^{n-1}{h_j(x,y,z)} - \sum_{j=1}^{n-1}{h_j(y,x,z)} - \left(\frac{1}{2}\right)^{n-1} \left(\frac{4}{5}\right)\\
    &\dspace\dspace\dspace\dspace+
    h_n(x,y,z) - h_n(y,x,z) + \left(\frac{1}{2}\right)^n \left(\frac{4}{5}\right)\\
    &= \greenbox{\dnr{n-1}(x,y,z) - \alpha_{n-1}} + \greenbox{h_n(x,y,z) - h_n(y,x,z)} + \alpha_n
\label{eqn:break-up-milp}
\end{align}

This points us to breaking the previous MILP into two MILPs: one minimizes ${\dnr{n-1}-\alpha_{n-1}}$ using the same formulation as \eqref{milp:origLP}\footnote{
Except we add the constant $-\alpha_{n-1}$ to the objective function. This does not affect the solution; it merely shifts all the objective values down by $\alpha_{n-1}$. 
}, and one minimizes ${h_n(x,y,z)-h_n(y,x,z)}$ using the following formulation:
\begin{milp}
\label{milp:dhLP}
    \begin{array}{ll@{}ll}
    \text{variables} & x,y,z, \{i: i \in \Hset_n(x,y,z) \cup \Hset_n(y,x,z)\}\\
    \text{minimize}  & h_n(x,y,z) - h_n(y,x,z) \\
    \text{subject to}& \indic{c} = 1 \Rightarrow c\\
                     & \indic{c} = 0 \Rightarrow \neg c\\
                     & 0<x<y<z 
    \end{array}
\end{milp}

Therefore, suppose the two MILPs output $\dn^*$ and 
\[
\beta^* := h_n(x^*,y^*,z^*)-h_n(y^*,x^*,z^*)
\] 
respectively, then if $\dn^* + \beta^* + \alpha_n > 0$, we have found the $n$-value at which we can terminate \meshalgo. This holds even if the $(x,y,z)$ solutions of these MILPs do not live in the same mesh-grid --- since these MILPs produce absolute lower-bounds on their objective function over all valid regions of $V$, the sum of their objective values must still be a lower-bound on $\dn$.

This approach has two uses:
\begin{enumerate}
    \item \textit{Verifying $\Delta_n>\alpha_n$}: Suppose $n^* = \min{\{n\geq 2: \Delta_n > \alpha_n\}}$. Due to the constraint $\Delta_n-\alpha_n>0$ in MILP \eqref{milp:origLP}, the MILP would be infeasible at $n=n^*$. We could then backtrack by one step, and use the pair of MILPs (\ref{milp:origLP} and \ref{milp:dhLP}), minimizing $\dnr{n^*-1}$ and ${h_{n^*}(x,y,z)-h_{n^*}(y,x,z)}$ respectively, to verify that $\dn>\alpha_n$ actually holds using Equation \eqref{eqn:break-up-milp}. This would additionally yield a numerical lower bound on $\dn$.
    \item \textit{Parallelizing computations}: This technique is an example of a \textit{divide-and-conquer} approach, and it motivates how one may parallelize the computation of $\meshalgo$: By further unrolling the recursion from $\dnr{n-1}$ (i.e., to $\dnr{n-2}, \dnr{n-3}$, etc), one can decompose the original MILP \eqref{milp:origLP} into several smaller, computationally less-demanding MILPs.
\end{enumerate}

\subsection{Generalizing the proof}
\label{sec:adapt-algo}
Here we extend the possible use cases of the $\meshalgo$ framework slightly. 

Suppose that one wanted to use $\meshalgo$ to try to prove an inequality of the form
\begin{equation}
    f(s_1,s_2,s_3) \oplus f(t_1,t_2,t_3)
\end{equation}
under $V$, where 
\begin{itemize}
    \item $\oplus \in \{<,>,\leq,\geq\}$ is a comparison operator,
    \item $s_1,s_2,s_3,t_1,t_2,t_3$ are linear combinations of $x,y,z$, and
    \item $V$ consists of inequalities of the form \eqref{eqn:linear-combo-constraint} only\footnote{Here, we no longer constrain $V$ to its definition in \eqref{eqn:V}.}.
\end{itemize}

Then one would have to make the following changes to the algorithm.
\begin{itemize}
    \item Replace the comparison operators with $\oplus$ accordingly: e.g., if $\oplus$ is $\geq$, then the inequality $\dn>\alpha_n$ in the original definition \eqref{alg:mesh} should be replaced with $\dn \geq \alpha_n$;
    \item Replace $(x,y,z)$ with $(s_1,s_2,s_3)$, and $(y,x,z)$ with $(t_1,t_2,t_3)$ in the definition of $\dn$ \eqref{eqn:dn};
    \item Replace the definition of input $V$ to the algorithm.
\end{itemize}

\section{Results and possible extensions}
\label{sec:results}
In this section we present the numerical results based on our Python \href{\srccode}{implementation} of \meshalgo, which completes the proof for Theorem \ref{thm:want_to_swap}.

\subsection{Completing the proof for Lemma \ref{lem:swap_xy1}}
\begin{itemize}
    \item The MILP for $\dn$ \eqref{milp:origLP} is infeasible at \boldmath$n=4$\unboldmath. 
    \item $\Delta_3(x,y,z) - \alpha_3 \geq -\dfrac{1}{60}$ for all regions in $V$.\footnote{This minimum value occurs at, e.g., $(x,y,z)=(4,5,6)$.}\\
    \item $h_n(x,y,z) - h_n(y,x,z) \geq -\dfrac{11}{648}$.
    \footnote{This minimum value occurs at, e.g., $(x,y,z)=(7,8,23)$.}
\end{itemize}
Putting these together as done in Equation \eqref{eqn:break-up-milp}:
\begin{align*}
    \Delta_4(x,y,z) - \alpha_4 
    &= \left(\Delta_3(x,y,z) - \alpha_3 \right) + \left( h_4(x,y,z) - h_4(y,x,z) \right) + \alpha_4 \\
    &\geq -\frac{1}{60} - \frac{11}{648} + \left(\frac{1}{2}\right)^4 \left(\frac{4}{5}\right)\\
    &=\frac{53}{3240} > 0\\
    \Rightarrow \Delta_4(x,y,z) 
    &\geq 
    \redbox{\frac{43}{648}}
    \approx 0.066358
\end{align*}
Since $\Delta_4(x,y,z) > \alpha_4$ for all subregions in $V$, the claim $f(x,y,z)>f(y,x,z)$ for $0<x<y<z$ follows.

\subsection{Completing the proof for Lemma \ref{lem:swap_yz1}}

It turns out that the point at which $\dn(y,z,x) > \alpha_n $ for all regions in $V$ is also \boldmath$n=4$\unboldmath.
Using the same proof structure and making the changes as explained in Section \ref{sec:adapt-algo}, with $(s_1,s_2,s_3) = (y,z,x)$, $(t_1,t_2,t_3)=(z,y,x)$ and $\oplus=\text{`$>$'}$, our results indicate that:
\begin{align*}
    \Delta_3(y,z,x) - \alpha_3 &\geq -\frac{23}{1080}\\
    h_4(y,z,x) - h_4(z,y,x) &\geq -\frac{7}{324}\\
    \Rightarrow \Delta_4(y,z,x) - \alpha_4 
    &\geq -\frac{23}{1080} - \frac{7}{324} + \left(\frac{1}{2}\right)^4 \left(\frac{4}{5}\right) \\
    &= \frac{23}{3240} > 0\\
    \Rightarrow \Delta_4(y,z,x) 
    &\geq 
    \redbox{\frac{37}{648}}
    \approx 0.057099
\end{align*}

Since $\Delta_4(y,z,x) > \alpha_4$ for all subregions in $V$, the claim $f(y,x,z)>f(z,x,y)$ similarly follows.

Figure \ref{fig:plots} visualizes how the minimum value of $\dn(x,y,z)$ (a) and $\dn(y,z,x)$ (b) varies over $V$, under a fixed-sum assumption on $x, y, z$.
\begin{figure}[H]
    \centering
    \includegraphics[width=1.4\linewidth]{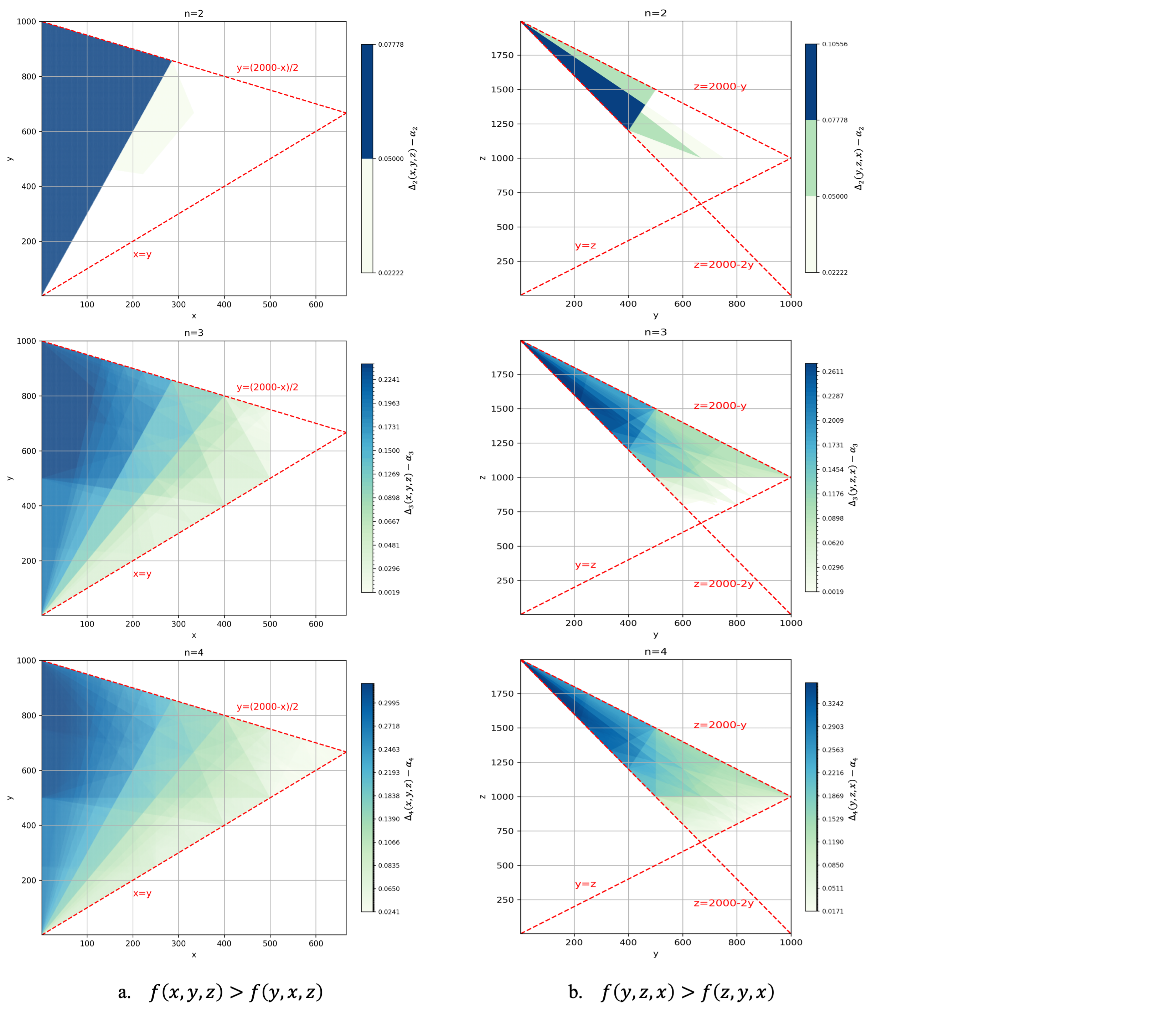}
    \caption{Visualization of ${\{ (x,y,z) : 0<x<y<z,\dspace x+y+z=2000,\dspace \Delta_n(\vect{s})>\alpha_n \}}$}
    {\begin{enumerate*}[a.\,]
        \item $\vect{s}=(x,y,z)$;\dspace\dspace\dspace
        \item $\vect{s}=(y,z,x)$
    \end{enumerate*}\\
    The colors indicate the values of $\Delta_n-\alpha_n$.}
    \label{fig:plots}
\end{figure}

\subsection{
%Conclusion and 
Future Work}
%\subsection{Summary}
%This paper is the first to systematically prove the seemingly intuitive result that, in the real-valued ``all in'' variant of the three-player Gambler's Ruin model, a player is less likely to lose if they swap their initial capital with another ``wealthier'' player. We accomplished this through a computer-assisted proof, beginning with the derivation of the general $\meshalgo$ algorithm, followed by a two-stage reduction to make $\meshalgo$ implementable. This involved transforming it first into a constraint optimization problem, and then into solvable mixed-integer linear program(s). Moreover, we described how $\meshalgo$ could be adapted to motivate proofs for other conjectures of a similar form to Theorem \ref{thm:want_to_swap}. This computer-aided framework underscores the significant potential of leveraging computational power in proving mathematical theorems.

%\subsection{Future work}
%There are several variations of Theorem \ref{thm:want_to_swap} that one could look at. We provide two directions below.
 
%\subsubsection{Points of inequality}
%Note that there exists $x<z$ such that 
%\begin{equation}
%    f(x,z,z) < f\left(\frac{x+z}{2}, \frac{x+z}{2}, z\right)
%\end{equation}

%Here is one possible result.
Angel and Holmes \cite{angel2024betting} also conjectured the following.
\begin{conjecture}
\label{conj:equal1}
For $x,y,z>0$ with $x<y$, 
\begin{equation}
f(x,y,z)>f(y,y,z).    
\end{equation}
\end{conjecture}

Our attempt to prove this conjecture using similar methods ran beyond $n=6$. For $n>6$, the program becomes quite slow due to the high branching factor (see Section \ref{sec:mip}). While our implementation already uses a degree of parallelism (over 8 cores), one could improve efficiency further by utilizing more computational resources, either by using a more powerful machine with additional processors, or through distributed parallel optimization across multiple machines \cite{bixby1995parallel,ladanyi2001branch}. We provided an initial direction for how this could be approached in Section \ref{sec:distr-opt}.

\subsubsection{Fixed sum}
A fixed sum $x+y+z$ could be of interest in resource allocation scenarios. This means $V$ would contain an extra constraint of the form, $x+y+z \oplus \ell$, where $ell \in \mathbb{R^+}$ is a constant and $\oplus \in \{ <,>,\leq,\geq \}$ is a comparison operator.

However, the second reduction described in Section \ref{sec:second-reduction} would no longer work for this case, since we would now have a constant term $ell$ present within the constraints of MILP \eqref{milp:origLP}. 
The exception is if $x,y,z \in \mathbb{Z^+}$: in the all-integer case, the program would still hold, as we can simply convert all strict inequalities in $x,y,z$ to an equivalent non-strict inequality (e.g., $x+y+z<\ell$ is equivalent to $x+y+z\leq \ell -1$).

%Here is one such result.
%\begin{theorem}
%\label{thm:equal1}
%For $z\ge 1/3$, and $0<x<y$ such that $x+y+z=1$, 
%\begin{align}
%    f(x,y,z)>f((1-z)/2,(1-z)/2,z)&=f((x+y)/2,(x+y)/2,z)\\
%    &>f(y,x,z).
%\end{align}
%\end{theorem}
%Let \[\Delta'_n(x,y,z):=\sum_{j=1}^n h_j(x,y,z)-\sum_{j=1}^n h_j((1-z)/2,(1-z)/2,z).\]
%Theorem \ref{thm:equal1} follows from the following two Lemmas, which we will verify using a computer-assisted proof.
%\begin{lemma}
%\label{lem:x<y1}
%For all $x,y,z>0$ with $z\ge 1/3$ and $0<x<y$ satisfying $x+y+z=1$, $\Delta'_6(x,y,z)\ge ?>(1/2)^6(4/5)$.
%\end{lemma}

%\begin{lemma}
%\label{lem:x<y2}
%For all $x,y,z>0$ with $z\ge 1/3$ and $0<x<y$ satisfying $x+y+z=1$, $-\Delta'_6(y,x,z)\ge ?>(1/2)^6(4/5)$.
%\end{lemma}

\newpage
\section*{Acknowledgement}
The authors thank Omer Angel, whose joint work with the second author laid the foundation for this project. 
%Without Prof. Angel's kind support for the opportunity to undertake this research, and his generous sharing of insights from their unpublished work, this project would not have been possible. 
%Throughout the project, he also posed insightful questions and advice, particularly on the computational aspects, which greatly enriched the depth and quality of our research.

\printbibliography

\newpage

\appendix
\section{Supplementary example: Computing $h_n$}
\label{sec:hrec-eg}
Using Equations \eqref{h1} and \eqref{hrecursion}, we could compute $h_2(x,y,z)$ as follows:
\begin{align*}
h_2(x,y,z) &= \frac16 \Big[h_{1}(x,2y,z-y)+h_{1}(2x,y-x,z)+h_{1}(2x,y,z-x) \Big]\\
&= \frac{1}{36}\Big[\indic{x\le 2y}+\indic{x\le z-y}+\indic{2x\le y-x}+\indic{2x\le z}+\indic{2x\le y}+\indic{2x\le z-x}\Big]
\end{align*}
Figure \ref{fig:hrec-eg} visualizes the recursion. Note that this is the reverse of the partial state diagram in Figure \ref{fig:dtmc}. The crossed out states are the nodes that we can prune early, under the constraint $V=\{(x,y,z):0<x<y<z\}$: e.g., $x-y<0$, so in state $(x-y, 2y, z)$, player 1 would have already been eliminated from the game, which makes its transition to $(x,y,z)$ impossible.
\begin{figure}[H]
    \centering
    \resizebox{1\textwidth}{!}{ \begin{tikzpicture}[
    every node/.style={draw, ellipse, minimum width=0.8cm, minimum height=0.4cm, font=\Large},  % Compact nodes with larger font
    every edge/.append style={->, -{Latex[length=2mm,width=3mm]}, thick, line width=1pt}
]

% Child nodes at the top
\node (1) at (-10,3) {\(2x,y-x,z\)};
\node (2) at (-6,3) {\(2x,y,z-x\)};
\node (3) at (-2,3) {\(x-y,2y,z\)};
\node (4) at (2,3) {\(x,2y,z-y\)};
\node (5) at (6,3) {\(x-z,y,2z\)};
\node (6) at (10,3) {\(x,y-z,2z\)};

% Main node (x,y,z) at the bottom
\node (xyz) at (0,0) {\(x,y,z\)};  % Increase font size for the main node

% Edges from nodes 1-6 to (x,y,z), with labels 1/6
\draw[->] (1) edge node[left, draw=none, below=0.1em] {\Large \(\frac{1}{6}\)} (xyz);
\draw[->] (2) edge node[above, draw=none, left=0.6em] {\Large \(\frac{1}{6}\)} (xyz);
\draw[->] (3) edge node[right, draw=none, above=1pt] {\Large \(\frac{1}{6}\)} (xyz);
\draw[->] (4) edge node[left, draw=none, above=1pt] {\Large \(\frac{1}{6}\)} (xyz);
\draw[->] (5) edge node[above, draw=none, right=0.6em] {\Large \(\frac{1}{6}\)} (xyz);
\draw[->] (6) edge node[left, draw=none, below=0.1em] {\Large \(\frac{1}{6}\)} (xyz);

% Red lines crossing out nodes 3, 5, and 6
\draw[red, thick] (-3.5, 3.8) -- (0.3, 2.5); % Node 3
\draw[red, thick] (4, 3.8) -- (7.8, 2.5); % Node 5
\draw[red, thick] (8, 3.8) -- (11.8, 2.5); % Node 6

\end{tikzpicture} }
    \caption{Parent states of $(x,y,z)$ in $h_n(x,y,z)$}
    \label{fig:hrec-eg}
\end{figure}
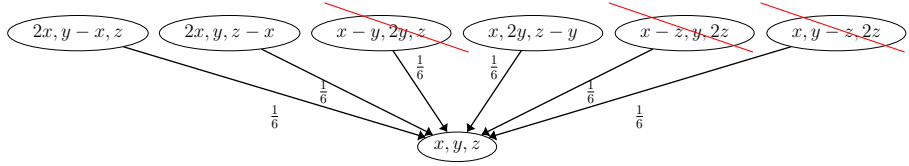

Similarly, we could compute $h_2(y,x,z)$ as follows:
\begin{align*}
h_{2}(y,x,z) &= \frac{1}{6}\Big[h_{1}(2y,x,z-y)+h_{1}(y-x,2x,z)+h_{1}(y,2x,z-x)\Big]\\
&= \frac{1}{36}\Big[\indic{2y\le x}+\indic{2y\le z-y}+\indic{y-x\le 2x}+\indic{y-x\le z}+\indic{y\le 2x}+\indic{y\le z-x}\Big]
\end{align*}

As introduced in Section \ref{sec:milp-repr}, we could represent $\Delta_2$ as $(c_2, \Hset_2)$, where:
\begin{align*}
    c_2 &= \frac{1}{36}\\
    \Hset_2^+ &= \{x \leq y, x \leq z-y, 2x \leq y-x, 2x \leq z, 2x \leq y, 2x \leq z-x\}\\
    \Hset_2^- &= \{2y \leq x, 2y\le z-y, y-x\le 2x, y-x\le z, y\le 2x, y\le z-x\}\\
    \Hset_2 &= \Hset_2^+ \cup \Hset_2^-
\end{align*}

\end{document}